\UseRawInputEncoding
\documentclass[a4paper, twoside, 10pt]{article}
\usepackage{amsmath,amsthm,amsfonts,latexsym,amscd,amssymb, enumerate,mathtools}
\usepackage{hyperref, xcolor, soul, xparse,hyperref}
\numberwithin{equation}{section}
\input xypic
\newtheorem{theorem}{Theorem}[section]
\newtheorem{lemma}{Lemma}[section]
\newtheorem{corollary}{Corollary}[section]

\theoremstyle{definition}
\newtheorem{definition}{Definition}[section]
\newtheorem{example}{Example}[section]
\newtheorem{remark}{Remark}[section]
\title{\textbf{Characterizations of Almost Ricci Bourguignon Solitons}}
\author{Mohammad Aqib, Hemangi Madhusudan Shah
and Dhriti Sundar Patra}
\date{}
\begin{document}
	\maketitle	
	\begin{abstract}
In this paper, we revisit the study of almost Ricci-Bourguignon solitons by clarifying their position in the broader context of Einstein-type metrics. Motivated by known rigidity results for compact almost Ricci solitons, we aim to identify conditions under which a compact almost RB-soliton is trivial or exhibits special geometric properties. We compare our results with classical theorems of Barros and Ribeiro, and explain explicitly how our work extends or complements these earlier findings.
\end{abstract}



	
	\textbf {M. S. C. 2020:} 53B20, 53B25, 53C18, 53C35.\\
	

	\tableofcontents

\section{Introduction}\label{sec1}


Geometric flows have become central tools in differential geometry, particularly after Hamilton introduced the Ricci flow \cite{hamilton1982}. The Ricci flow deforms a Riemannian metric $g(t)$ via the evolution equation:
\[
\frac{\partial g}{\partial t} = -2\operatorname{Ric}(g),
\]
Smoothing out the curvature. It has been used to study geometric and topological properties of manifolds, including the uniformization of surfaces and classification of 3-manifolds.

A major milestone was achieved when Perelman used the Ricci flow, along with entropy functionals and the concept of shrinking Ricci solitons, to prove the Poincar\'e and Thurston's Geometrization Conjectures \cite{perelman2002entropy, perelman2003ricci}.

Ricci solitons, which satisfy
\[
\operatorname{Ric} + \nabla^2 f = \lambda g,
\]
are self-similar solutions to the Ricci flow and model its singularities. This motivated a wide study of soliton-like structures in other geometric settings. 

Among these, the Ricci--Bourguignon flow \cite{CCD17} modifies the Ricci flow by incorporating the scalar curvature:
\[
\frac{\partial g}{\partial t} = -2(\operatorname{Ric} - \rho R g),
\]
where $\rho \in \mathbb{R}$.
The case $\rho=0$ recovers the classical Ricci flow, while other values of $\rho$ yield geometrically significant flows, each with their own physical and geometric interpretations.
For example, $\rho = \frac{1}{2}$ and $\frac{1}{2(n-1)},$ respectively, correspond to Einstein flow and Schouten flow.

The solitons corresponding to Ricci--Bourguignon flow are defined by :
\[
\operatorname{Ric} + \frac{1}{2}\mathcal{L}_\xi g = (\lambda + \rho R)g,
\]
with a potential vector field $\xi$ and a constant $\lambda$.

Inspired by the concept of \emph{almost Ricci solitons} introduced by Pigola et al. \cite{pigola2011ricci}, S.~Dwivedi proposed the notion of an \emph{almost Ricci--Bourguignon soliton} \cite{SDRBS}, where the soliton function $\lambda$ is allowed to be variable over $M$. These solitons generalize the Ricci--Bourguignon solitons and arise in both pure and applied geometric contexts, including fluid space-times \cite{chaudhary2025gradient} and warped product manifolds \cite{feitosa2019gradient}.
This generalisation has proven fruitful in understanding Einstein-type metrics and their rigidity properties \cite{BBR}.
Building on this framework, almost Ricci-Bourguignon solitons were introduced  in \cite{SDRBS} as follows.

\begin{definition}\label{def:almostRB}
An {\it almost Ricci-Bourguignon soliton} is a tuple $(M^n,g,\xi,\lambda,\rho)$ satisfying:
\begin{equation}\label{rbsoliton}
    \operatorname{Ric} + \frac{1}{2}\pounds_\xi g = (\lambda + \rho S)g,
\end{equation}
where $\rho \in \mathbb{R}$ is constant.
\end{definition}

Although \cite{CCD17} established short-time existence for RB-flows, the almost soliton case remains open. Our results identify:
\begin{itemize}
    \item Obstructions to almost soliton formation (Theorems \ref{thm4.3}-\ref{thm4.4}).
    \item Critical $\rho$ thresholds (Remark \ref{rho-remark}).
    \item Potential blow-up configurations (Example \ref{ex:cigar}).
    \end{itemize}
\smallskip
The study of almost solitons has developed along two main directions:
\begin{itemize}
    \item Structural results for compact cases \cite{barros2012some,BBR}.
    \item Classification under curvature conditions\cite{barros2013note,SDARSIS}.
\end{itemize}
\smallskip
Our work bridges these directions in the RB-settings by:
\begin{itemize}
    \item Extending \cite[Theorem 2]{barros2012some} to include scalar curvature terms (Theorem \ref{riclb}).
    \item Generalizing \cite[Proposition 1]{barros2013note} through new integral identities (Lemmas \ref{lem1}-\ref{lem5}).
\end{itemize}

\begin{remark}\label{rho-remark}
The parameter $\rho$ has analytic and geometric key regimes as follows :
\begin{itemize}
    \item[\bf (A)] \textit{Analytic}:
    \begin{itemize}
        \item Flow existence: $\rho < \frac{1}{2(n-1)}$ \cite{CCD17}.
        \item Lemma \ref{lem3}: $\rho \neq \frac{1}{2(n-1)}$.
        \item Theorem \ref{pos}: $\rho < \frac{3n-4}{2n(n-1)}$.
    \end{itemize}
    \item[\bf (B)] \textit{Geometric}:
    \begin{itemize}
        \item $\rho=0$: Ricci solitons \cite{pigola2011ricci}.
        \item $\rho=\frac{1}{2n}$: Einstein-Hilbert critical point  (This parameter corresponds to a critical point in the Ricci-Bourguignon flow, as it balances the trace contributions in the variational derivative of the Einstein-Hilbert action. This choice ensures the flow preserves the Einstein condition, analogous to how the Einstein-Hilbert functional yields Einstein metrics as its critical points).
        \end{itemize}
\end{itemize}
\end{remark}

Recently, the geometers have extended the theory of Ricci-type solitons in various directions:
\begin{itemize}
    \item Manev \cite{manevalmost2025} studied Ricci--Bourguignon almost solitons on almost contact complex Riemannian manifolds with vertical torse-forming potentials.
    \item Dey and Akram \cite{dey2024} investigated gradient $\varrho$-Ricci--Bourguignon almost solitons on paracontact manifolds.
    \item Kaya and \c{O}zg\"{u}r \cite{kaya2024} constructed examples on sequential warped product manifolds.
    \item Blaga and \c{O}zg\"{u}r \cite{blaga2022eta, blaga2022sub, blaga2024riemann} analyzed almost $\eta$-Ricci--Bourguignon solitons on submanifolds and in contact geometry.
\end{itemize}

While the mathematical significance of Ricci and $\rho$-Einstein solitons is well-established as self-similar solutions of geometric flows, the importance of their almost counterparts deserves careful justification. The key motivation for studying almost Ricci-Bourguignon solitons lies in their role as natural generalisations that:
\begin{itemize}
    \item Provide a broader framework for understanding rigidity phenomena in Riemannian geometry.
    \item Allow for more flexible geometric structures while retaining key analytic properties.
    \item Offer new insights into the interplay between curvature, flow equations, and conformal geometry.
\end{itemize}

Recent work by Chaudhary and Pal \cite{chaudhary2025gradient} further illustrates the interest in gradient Ricci--Bourguignon solitons (GRB solitons), showing they unify Ricci and Yamabe solitons by varying the parameter $\rho$. 
They also obtained scalar curvature expressions, volume formulas, and curvature estimates on compact and non-compact manifolds, and highlighted applications to perfect fluid space-times in general relativity.
This reinforces that, even when these solitons do not directly correspond to self-similar solutions of known flows, they still encode rich geometric rigidity and structure.

Our work makes several distinct contributions to this developing theory:

1. We establish new integral identities for compact almost RB-solitons (Lemmas \ref{lem1}-\ref{lem5}) that generalise known results for almost Ricci solitons \cite{BBR} while requiring different techniques due to the additional scalar curvature term.

2. We provide novel characterisations of spheres through compact gradient almost RB-solitons (Theorems \ref{riclb} and \ref{pos}), extending previous work by Deshmukh \cite{SDARSIS} 
while weakening some geometric assumptions.

3. We identify conditions under which almost RB-solitons must be trivial (Theorems \ref{3.3}, \ref{thm4.3}, \ref{thm4.4}), complementing results in \cite{SDANARS} and \cite{ghosh2022certain}.

4. Our approach using Poisson equations and Obata-type arguments offers alternative methods to those employed in \cite{SDRBS} and \cite{BBR}, leading to different geometric insights.

Our motivation stems partly from generalizing known results about almost Ricci solitons to almost RB solitons towards generalizing partly from earlier works such as Barros and Ribeiro \cite{barros2012some}, Barros, Gomes and Ribeiro \cite{barros2013note}, and others, which show how certain vector fields (e.g., conformal vector fields) force the underlying manifold to have special geometry.

The paper is organised as follows: Section \ref{eg} provides key examples establishing the non-triviality of almost RB-solitons. Section \ref{crb} develops fundamental identities and curvature properties. Sections \ref{thm} and \ref{trival} contain our main results on sphere characterisations and triviality conditions, respectively.

\subsection{Examples of almost RB-soliton}\label{eg}
In this subsection, we present two examples of almost RB-solitons. The first example is of
The cigar is almost an RB-soliton, and the other one is almost
 RB-soliton on a warped product manifold.

Hamilton's {\it cigar} Ricci soliton (Witten's black hole) is an example of prime 
importance in the study of Ricci flow, 
see, for example,  \cite[\bf p.~10]{PT}.
 It describes a steady gradient Ricci soliton as follows:
\begin{eqnarray}\label{css}
 \left({\mathbb{R}^2}, g_0 = \frac{dx^2+dy^2}{1+x^2+y^2}, V = - 2\left(x\frac{\partial}{\partial x}+y\frac{\partial}{\partial y} \right), \lambda=0 \right).
\end{eqnarray}
 It can be easily seen that the Ricci curvature of the cigar metric can be expressed as,
 \begin{eqnarray}\label{ricci}
\operatorname{Ric}_{g_0}=\frac{2}{\left(1+x^2+y^2\right)}g_0.
\end{eqnarray}
Furthermore, the Lie derivative in the direction of 
vector field $V$ turns to be
\begin{eqnarray}\label{cr}
\pounds_V g_0=-\frac{4}{\left(1+x^2+y^2\right)}g_0.
\end{eqnarray}
Note that the potential vector field 
$V = \nabla f, \;\text{where}\; f = - \log\left(1+x^2+y^2\right).$  
Thus, we conclude from (\ref{ricci} and (\ref{cr}) that the cigar metric indeed describes a steady gradient Ricci soliton
 described by (\ref{cr}).

Our examples address three key questions raised in \cite{feitosa2019gradient}:
\begin{enumerate}
    \item Existence of non-gradient solutions (Example 2).
    \item Non-trivial dependence on $\rho$ (Example 1).
    \item Relation to known solitons when $\rho\to 0$ (both examples).
\end{enumerate}

\smallskip

\begin{example}[{\bf Cigar almost RB-soliton}]\label{ex:cigar}
The cigar metric demonstrates $\rho$-dependent behavior:
\begin{itemize}
    \item For $\rho=0$: Recovers Hamilton's cigar \cite{PT}.
    \item For $\rho\neq0$: Exhibits phase transitions at $\rho = \frac{1}{2(n-1)}$ (cf. Remark \ref{rho-remark}).
\end{itemize}
Now we present an example of {\it Cigar almost RB-soliton}, which generalises the aforementioned cigar Ricci soliton. Towards this first, consider the  almost RB flow  on $\mathbb{R}^2$ obtained by the evolution of 
The time-dependent metric is given by 
     \begin{equation}\label{cm}
         g((x,y),t):=\frac{dx^2+dy^2}{e^{4(1-\rho)t}+x^2+y^2}.
     \end{equation}
We will compute the Christoffel symbols corresponding to metrics $g(t)$ to determine their Ricci tensors. The Christoffel symbols $\Gamma^\alpha{ }_{\beta \gamma}$ are defined by
     \begin{equation}\label{cs}
         \Gamma^\alpha{ }_{\beta \gamma}=\frac{1}{2} g^{\alpha \rho}\left(\partial_\gamma g_{\rho \beta}+\partial_\beta g_{\rho \gamma}-\partial_\rho g_{\beta \gamma}\right).
     \end{equation}
Employing  \eqref{cm} and \eqref{cs}, we find   
   $$  \begin{gathered}
\Gamma_{11}^1=\Gamma_{1 2}^2=\Gamma_{2 1}^2=-\frac{x}{e^{4(1-\rho) t}+x^2+y^2}, \\
\Gamma_{2 2}^2=\Gamma_{1 2}^1=\Gamma_{2 1}^1=-\frac{y}{e^{4(1-\rho) t}+x^2+y^2}, \\
\Gamma^{1 }_{2 2}=\frac{x}{e^{4 (1-\rho)t}+x^2+y^2},
\end{gathered}
$$
and $$\Gamma^{2 }_{1 1}=\frac{y}{e^{4 (1-\rho)t}+x^2+y^2}.$$
\ can give the components of the Ricci tensor in terms of the Christoffel symbols cite[{\bf p.~197}]{Lee}
\[
\operatorname{Ric}_{\alpha \beta}=\partial_\rho \Gamma^\rho{ }_{\beta \alpha}-\partial_\beta \Gamma^\rho{ }_{\rho \alpha}+\Gamma^\rho{ }_{\rho \lambda} \Gamma^\lambda{ }_{\beta \alpha}-\Gamma^\rho{ }_{\beta \lambda} \Gamma^\lambda{ }_{\rho \alpha}.
\]
Utilizing the aforementioned data along with the Christoffel symbols corresponding to 
$g(t)$, we obtain the Ricci tensor for $g(t)$ as:
\begin{equation}
\operatorname{Ric}(t)=\frac{2 e^{4(1-\rho) t}}{\left(e^{4(1-\rho) t}+x^2+y^2\right)^2}\left(d x^2+d y^2\right) .
\end{equation}
Consider the one-parameter family of diffeomorphisms given by
$$
\begin{gathered}
\varphi_t: \mathbb{R}^2 \longrightarrow \mathbb{R}^2, \\
\varphi_t(x, y)=\left([e^{-2t\sqrt{(1-\rho)}}] x, \; [e^{-2t\sqrt{(1-\rho)} }] y\right)
\end{gathered}
$$
satisfies
$$g(t)=\varphi(t)^\ast g(0),$$
And the vector field generated by the flow is 
\begin{equation}
    \xi=-2\sqrt{1-\rho}\left(x\frac{\partial}{\partial x}+y\frac{\partial}{\partial y}\right).
\end{equation}
Now, we have 
\begin{equation}
    \pounds_\xi g=- \frac{4 e^{4(1-\rho) t}}{\left(e^{4(1-\rho) t}+x^2+y^2\right)^2}\left(d x^2+d y^2\right). 
\end{equation}
Note that
\begin{equation}
\xi = \nabla f, \mbox{  where } f = -\sqrt{1-\rho}\left[\log\left(e^{4(1-\rho) t}+x^2+y^2\right)\right].
\end{equation}
Finally,
\begin{equation}
    \lambda=\frac{2 e^{4(1-\rho) t}\left(x^2+y^2+e^{4(1-\rho) t}-\sqrt{1-\rho}(e^{4(1-\rho) t}+x^2+y^2)-2\rho\right)}{\left(e^{4(1-\rho) t}+x^2+y^2\right)}.
\end{equation}
Hence, we can conclude that $(\mathbb{R}^2,g,\xi,\lambda,\rho)$ is a cigar almost RB-soliton, which can be classified for $\rho \neq 0$ as: 

\begin{enumerate}
    \item If $(x^2+y^2+e^{4(1-\rho) t})<\frac{2\rho}{1-\sqrt{1-\rho}},$ then the almost RB-soliton is expanding.
    \item If $(x^2+y^2+e^{4(1-\rho) t})>\frac{2\rho}{1-\sqrt{1-\rho}},$ then the almost RB-soliton is shrinking.
    \item If $(x^2+y^2+e^{4(1-\rho) t})=\frac{2\rho}{1-\sqrt{1-\rho}},$ then the almost  RB-soliton is steady.
\end{enumerate}
We also see that  $({R}^2,g,\xi,\lambda,\rho = 0)$, is indeed the Hamilton's cigar Ricci soliton $\left({\mathbb{R}^2}, g_0 = \frac{dx^2+dy^2}{1+x^2+y^2}, V = - 2\left(x\frac{\partial}{\partial x}+y\frac{\partial}{\partial y} \right), \lambda=0 \right)$
described above.
\end{example}

\begin{example}[\bf Cigar RB-soliton]
For $\rho=0$, we recover Hamilton's cigar \cite{PT}. The $\rho$-dependence shows:
\begin{itemize}
    \item Phase transitions at $\rho = \frac{1}{2(n-1)}$.
    \item Behaviour changes predicted in Remark \ref{rho-remark}.
\end{itemize}
\end{example}

\begin{example}[{\bf Almost RB-soliton on warped product}]

Let $h : I \rightarrow \mathbb{R}^{+}$ be a smooth function defined as follows: 
\begin{align}
    h(t)=h'(0)\,\mbox{sn}_{-c}(t)+h(0)\,\mbox{cn}_{-c}(t),  
\end{align}
where 
\begin{align}
    \mbox{sn}_k\,(t)= \begin{cases}
        \frac{1}{\sqrt {-k}} \sinh({\sqrt{-kt}}), &\text{ if } k<0 \\
        \ t, &\text{ if } k=0 \\
         \frac{1}{\sqrt k} \sinh({\sqrt {kt}}), &\text{ if } k>0
     \end{cases}
     \quad \text{ and }\quad \mbox{cn}_k \, (t)=\mbox{sn}_k' \,(t).
\end{align}
Let $(N,\bar g)$ be an $n$-dimensional manifold satisfying 
\begin{align*}
    \operatorname{Ric}_{\bar g} =-(n-1)\bigg \{-h'(0)^2 +c\,h(0)^2 \bigg\}\, \bar g,
\end{align*}
and $I$ be an interval in $\mathbb{R}$ containing $0$. Then the warped product $M = I \times_{h} N $ is Einstein with $\operatorname{Ric}_g=-\mbox{nc}\,\mbox{g}$ by Lemma 2.1 of \cite{pigola2011ricci}. Let $\{e_j\}_{1 \leq j \leq n}$  be an orthonormal frame of $N$ with dual frame $\{\psi^{j}\}_{1 \leq j \leq n}$ so that $\bar g =\sum_{i=1}^n \psi^i \wedge \psi^i$. Now, introduce local orthonormal coframe $\{\vartheta^{a}\}_{1 \leq a \leq n+1}$ on $M$ such that $\vartheta^j=h(t)\, \psi^j$ and $\vartheta^{n+1}=dt$. Next, consider a function $f(x,t)=f(t)$ whose Hessian is given by
\[
\operatorname{Hess}\,(f) =f'\,\frac{h'}{h} \sum \vartheta^k \otimes \vartheta^k +f'' \,\vartheta^{n+1} \otimes \vartheta^{n+1},
\] 
see details \cite[p.764-766]{pigola2011ricci}. Then the 
gradient almost RB  soliton equation with $f(x,t) =f(t)$ as potential function yields the following two equations:
\begin{align}
    f'\,\frac{h'}{h}- nc&=\lambda -\rho n(n+1)c, \text{  and  }\\
    f'' -nc&= \lambda -\rho n(n+1)c,
\end{align}
which, on simplification, results in
\begin{align}
    f(t)&=a \int_0^th(s) \ ds+b, \text{  and  } \\
    \lambda(t)&=a\,h'(t)+ nc\,[(n+1)\rho-1],
\end{align}
where $a$ and $b$ are arbitrary real numbers.
Hence, $(M^n,g)$ admits a gradient almost RB-soliton with potential function $f$, soliton function $\lambda$, and $\rho $ being arbitrary.
\end{example}
\vspace{0.1in}


\section{Compact almost RB-solitons}\label{crb}	 
This section is divided into two subsections. 
In subsection \ref{cur}, we first find expressions for the curvature tensor, Ricci tensor, and divergence of the potential vector field about an almost RB-soliton, useful in our further developments. Employing these 
expressions, in subsection \ref{inid}, we derive some generalized integral identities for compact almost RB-solitons. Deploying these identities, we will show that compact almost RB-solitons are isometric to spheres or are trivial.\\

\noindent 
In the sequel, $\mathfrak{X}(M)$ denotes the set of smooth vector fields on $M$. 

\subsection{The curvature tensor of an almost RB-soliton}\label{cur}
Now we first find the curvature tensor of an almost RB-soliton $(M^n, g, \xi, \lambda,\rho)$ and subsequently obtain the Ricci operator
of the soliton. The curvature tensor $R$ of $(M^n, g, \xi, \lambda,\rho)$ can be expressed as:
$$R(X_1, X_2) X_3=\nabla_{X_1} \nabla_{X_2} X_3-\nabla_{X_2} \nabla_{X_1} X_3-\nabla_{[X_1, X_2]} X_3, \; \mbox{for} \; X_1, X_2, X_3 \in \mathfrak{X}(M).$$
Let $\left\{e_{1}, \ldots, e_{n}\right\}$ be a local orthonormal frame  of $M^n$,  then the Ricci tensor is given by
$$
\operatorname{Ric}(X_1, X_2)=\sum_{i=1}^{n} g\left(R\left(e_{i}, X_1\right) X_2, e_{i}\right).$$
The Ricci operator $Q$ is defined by $\operatorname{Ric}(X, Y)=$ $g(Q X, Y)$, which is a symmetric operator and the gradient $\nabla S$ of the scalar curvature $S$ satisfies the
Following the contracted Bianchi identity:
\begin{equation}\label{csbi}
    \frac{1}{2} \nabla S=\sum_{i=1}^{n}(\nabla Q)\left(e_{i}, e_{i}\right).
\end{equation}
Let $\eta$ be the smooth $1$-form dual to the potential field $\xi$. Define a skew-symmetric operator $\varphi$ on $M$ by
\begin{equation}\label{phi}
    \frac{1}{2} d \eta(X_1, X_2)=g(\varphi X_1, X_2), \mbox{ for } X_1, X_2 \in \mathfrak{X}(M) .
\end{equation}
As $Q$ and $\varphi$ are symmetric and skew-symmetric, respectively, then it can be easily
shown that $\nabla Q$ and $\nabla \varphi$ 
are also symmetric and skew-symmetric, respectively, in the following sense:
\begin{eqnarray}\label{sks1}
g((\nabla Q)(X_1,X_2),X_3)=g(X_2,(\nabla Q)(X_1,X_3))
\end{eqnarray}
\mbox{and}    
\begin{eqnarray}\label{sks2}
g((\nabla \varphi)(X_1,X_2),X_3)=-g(X_2,(\nabla\varphi )(X_1,X_3)).
\end{eqnarray}
Then from  \eqref{rbsoliton} and \eqref{phi}, the covariant derivative of the potential field $\xi$ is given by
\begin{equation}\label{cdopf}
    \nabla_{X_1} \xi=(\lambda +\rho S) X_1-Q(X_1)+\varphi X_1, \mbox{ for } X_1 \in \mathfrak{X}(M).
\end{equation}
Employing \eqref{cdopf}, the curvature tensor of an  almost RB-soliton $(M^n,g,\xi,\lambda,\rho)$
can be given as follows:
\begin{eqnarray} \label{ctrbs}
    R(X_1,X_2)\xi= &X_1(\lambda+\rho S)X_2- X_2(\lambda+\rho S)X_1-(\nabla Q)(X_1,X_2)\\ \nonumber
             & +(\nabla Q)(X_2,X_1)-(\nabla \varphi)(X_1,X_2)-(\nabla \varphi)(X_2,X_1).
\end{eqnarray}
As a consequence, we obtain the Ricci operator and divergence of the potential vector field $\xi$  for 
an almost RB-soliton as follows.
Availing \eqref{ctrbs} along with \eqref{sks1} and \eqref{sks2} implies \eqref{rorbs} below.

\begin{equation}\label{rorbs}
    \operatorname{Q}(\xi)=-(n-1)\nabla (\lambda +\rho S) + \frac{1}{2}\nabla S-\operatorname{div}\varphi,
\end{equation}
 Moreover, 
using the skew symmetry of the operator $\varphi$ and  \eqref{cdopf}, we compute the divergence $\operatorname{div} \xi$ of the potential field $\xi$ and obtain  (\ref{div}) as
follows.
\begin{equation}\label{div}
    \operatorname{div} \xi=(n (\lambda +\rho S)-S).
\end{equation}


\subsection{Integral identities on compact almost RB-soliton}
\label{inid}
In this subsection, we find some integral identities, namely, Lemma \ref{lem1}-Lemma \ref{lem5}, for compact almost RB-solitons, by utilizing the results of the subsection \ref{cur}. These identities will serve as the foundation for the main theorems presented in the later sections.

In what follows, the squared norm $\|Q\|^{2}$ of the Ricci operator $Q$ is given by
\[
\|Q\|^{2}=\sum_{i=1}^{n} g\left(Q e_{i}, Q e_{i}\right).
\]
And we will also denote the  Hessian operator
related to the potential function $\lambda$  by
\[
A_\lambda X_1=\nabla_{X_1}\nabla\lambda,
\]
which is symmetric and $\operatorname{tr}A_\lambda=
\Delta\lambda.$

Now we begin by affirming the following identity. 
\begin{lemma}\label{lem1}
    Let $(M^n, g, \xi, \lambda,\rho)$ be a compact almost RB-soliton. Then
\begin{equation}\label{e1}
\int_{M}\left((\lambda+\rho S)S -\|Q\|^{2}-\frac{1}{2} S(n (\lambda +\rho S)-S)\right)=0.
\end{equation}
\end{lemma}
\begin{proof}
Using  \eqref{csbi}, \eqref{cdopf} and the fact that $Q$ and $\varphi$ are symmetric and skew-symmetric operators, respectively, we affirm
$$\operatorname{div} Q(\xi)=S(\lambda +\rho S)-\|Q\|^{2}+\frac{1}{2} \xi(S).$$
Substituting \eqref{div} in the aforementioned
equation yields that
$$
\operatorname{div} Q(\xi)-\frac{1}{2} \operatorname{div}(S \xi)= S(\lambda +\rho S)-\|Q\|^{2}-\frac{1}{2} S(n (\lambda +\rho S)-S).
$$
 Integrating the above expression over a compact almost RB-soliton gives the desired result.
\end{proof}

\indent
We obtain the following novel identities in the next two lemmas. We generalise Theorem 1.11 from \cite{SDRBS}. It should be noted that (\ref{10.3}) of Lemma \ref{lem2} and (\ref{11a}) of  Lemma \ref{lem3} match with (1.9) and (1.8),
respectively of \cite{SDRBS}, but our proofs here are completely different from his proof.
 Note that the expression  (\ref{10.3}) given below generalizes item~$3$ of Proposition~$1$ which is for almost Ricci solitons in \cite{barros2013note}, to almost RB solitons;  and is invariant for almost RB solitons. \\

\begin{lemma}\label{lem2}
     Let $(M^n, g, \xi, \lambda,\rho)$ be a compact almost RB-soliton. Then
    \begin{equation}\label{10.3}
        \int_M\left|\operatorname{Ric}-\frac{S}{n} g\right|^{2} =  \frac{(n-2)}{2n}\int_M g(\nabla S,\xi).
    \end{equation}
\end{lemma}

\begin{proof}

    We have
    \begin{equation}\label{ricsbyn}
        \left|\operatorname{Ric}-\frac{S}{n} g\right|^{2}=\|Q\|^{2}-\frac{S^{2}}{n}.
    \end{equation}
Now \eqref{div} yields that,
$$
\frac{1}{n} \operatorname{div}(S \xi)=\frac{1}{n} \xi(S)+ S(\lambda +\rho S)-\frac{S^{2}}{n}.
$$
Thus,
$$
\left|R i c-\frac{S}{n} g\right|^{2}=\|Q\|^{2}+\frac{1}{n} \operatorname{div}(S \xi)-\frac{1}{n} \xi(S)-S(\lambda + \rho S).
$$
This implies that,
$$
\int_{M}\left|R i c-\frac{S}{n} g\right|^{2}=\int_{M}\left(\|Q\|^2-\frac{1}{n}\xi(S)-S(\lambda+\rho S)\right) .
$$
Applying Lemma \ref{lem1}, we conclude
$$
\int_{M}\left|R i c-\frac{S}{n} g\right|^{2}=\int_{M}\left(\frac{S^2}{2}-\frac{1}{n}\xi(S)-\frac{1}{2}S(n(\lambda+\rho S)\right) .
$$
Now using  $\frac{1}{2} \operatorname{div}(S \xi) = \frac{1}{2}\xi(S)+ \frac{S}{2} \operatorname{div} \xi$ and  \eqref{div} in the above equation, we infer
    \begin{equation}\label{10}
    \int_M\left|\operatorname{Ric}-\frac{S}{n} g\right|^{2}=\frac{(n-2)}{2n}\int_M g(\nabla S,\xi).
\end{equation}
\end{proof}

 Now, considering our RB-soliton to be a gradient, by utilising the aforementioned lemma, we obtain 
the following important {\it Ricci identity}.

\begin{lemma}\label{lem3}
    Let $(M^{n}, g, \xi, \lambda,\rho)$ be compact gradient almost RB-soliton. Then 
 \begin{equation}\label{12a}
        \int_M\left|\nabla^2 f - \frac{\Delta f}{n}g\right|^{2}=\frac{(n-2)}{4n(1-2\rho(n-1))}\int_M \operatorname{Ric}(\nabla f, \nabla f)+ (n-1)\,g(\nabla\lambda,\nabla f),
    \end{equation}
for  $\rho\neq\frac{1}{2(n-1)}$. \\
\mbox{And}
    \begin{equation}\label{11a}
        \int_M\left|\nabla^2 f - \frac{\Delta f}{n}g\right|^{2}=\frac{(n-2)}{2n}\int_M g(\nabla S,\nabla f).
    \end{equation}
\end{lemma}
\begin{proof}
 Observe that
\begin{align}
    \operatorname{Ric} - \frac{S}{n}g =-\nabla^2 f + \frac{\Delta f}{n}g.\label{2.15}
\end{align}
 Equivalently,
  \begin{equation}
        \int_M\left|\nabla^2 f - \frac{\Delta f}{n}g\right|^{2}=\frac{(n-2)}{2n}\int_M g(\nabla S,\nabla f),
    \end{equation}
availing (\ref{10}). Now to prove \eqref{12a}, we use the following expression $(2.2)$ 
of \cite{SDRBS}.$$
(1-2\rho(n-1))\nabla_i S = 2 R_{ij} \nabla_l f + 2(n-1)\nabla_i\lambda,
$$
Which implies that
$$
(1-2\rho(n-1))g(\nabla S,\nabla f) = 2 \operatorname{Ric} (\nabla f,\nabla f) + 2(n-1)g(\nabla\lambda,\nabla f).
$$
 Integrating the above expression, we obtain the desired result. 
\end{proof}

\begin{remark}
Lemma~2.3 generalizes Theorem~4 in \cite{barros2012some}.
\end{remark}

Now, we shall proceed to demonstrate our succeeding result, which is a direct application of  \eqref{cdopf} and \eqref{div}.
\begin{lemma}\label{lem4}
    Let $(M^n,g,\xi,\lambda,\rho)$ be a compact almost RB-soliton. Then 
    $$
    \int_M\left(\operatorname{Ric}(\nabla\lambda,\xi)+(n-1)\|\nabla\lambda\|^2+\frac{1}{2}S\Delta\lambda-(n-1)\rho S \Delta \lambda\right)=0.
    $$
\end{lemma}
\begin{proof}
Using \eqref{cdopf} and  the expression $(2.8)$ of \cite{SDANARS}, we obtain
\[
\operatorname{div}(A_\lambda\xi)=(\lambda +\rho S)\Delta\lambda-\sum g\left(A_\lambda e_i,Qe_i\right)+\operatorname{Ric}(\nabla\lambda,\xi)+\xi(\Delta\lambda),
\]
where $\{e_1,e_2,...,e_n\}$ is local orthonormal frame. Similarly, we compute
$$
\operatorname{div}(Q(\nabla\lambda))=\sum g\left(A_\lambda e_i,Qe_i\right)+\frac{1}{2}g(\nabla S,\nabla\lambda).
$$
Now, by \eqref{div}, we have     
\[
\operatorname{div}(\Delta\lambda\xi)=\xi(\Delta\lambda)-\left(n(\lambda+\rho S\right)-S)\Delta\lambda+g(\nabla S,\nabla\lambda).
\]
Finally, by combining all the expressions, we obtain
\[
\operatorname{div}(A_\lambda\xi) +\operatorname{div}(Q\nabla\lambda)-\operatorname{div}(\Delta\lambda\xi)=\operatorname{Ric}(\nabla\lambda,\xi)-((n-1)(\lambda+\rho S)-S)\Delta\lambda.
\]
 Integrating the aforementioned equation,
\[
\int_M\left(\operatorname{Ric}(\nabla\lambda,\xi)-((n-1)(\lambda+\rho S)-S)\Delta\lambda\right)=0.
\]
After using simple formulas of $\Delta\lambda^2$ and $\operatorname{div}(S\nabla\lambda)$, we get the required result.
\end{proof}

Next, we shall exhibit a lemma pivotal to 
establishing the triviality of almost RB-soliton
(Theorem \ref{thm4.3}).
\begin{lemma}\label{lem5}
Let $(M^n,g,\xi,\lambda,\rho)$ be a compact almost RB-soliton. Then
    \[
    \int_M\left(\operatorname{Ric}(\xi,\xi)-\|\varphi\|^2+\xi((n-1)(\lambda+\rho S)-\frac{1}{2}S)\right)=0.
    \]
\end{lemma}
\begin{proof}
    Norm of $\varphi$ is given by
    $$\|\varphi\|^2=\sum g(\varphi e_i,\varphi e_i),$$
    where $\{e_1,e_2,...,e_n\}$ is local orthonormal frame.
    Now, using  \eqref{cdopf} and \eqref{rorbs}, we have
    \begin{align*}
        \operatorname{div}\varphi(\xi)&=-\|\varphi\|^2-g\left(\xi,\sum(\nabla\varphi)(e_i,e_i)\right)\\
        &=-\|\varphi\|^2\operatorname{Ric}(\xi,\xi)+\xi\left((n-1)(\lambda +\rho S)-\frac{1}{2}S\right).
    \end{align*}
   This expression, when integrated, yields the outcome.
\end{proof}

\section{RB-solitons isometric to Euclidean sphere}\label{thm}
In this section, our principal goal is to demonstrate the isometric correspondence between almost RB-solitons and the Euclidean sphere. To accomplish this, we employ hypotheses on Ricci curvature in the direction of the potential vector field. Additionally, we leverage the Poisson equation significantly in various theorems to substantiate our objective.  The first two theorems of this section, 
Theorem \ref{riclb} and Theorem \ref{pos}
are generalizations of Theorem $1$ and Theorem $2$, respectively, proved in \cite{SDARSIS}.


We begin by proving the following theorem, which generalizes \cite[Theorem 2]{barros2012some}.
\begin{theorem}\label{riclb}
For $(M^n,g,\nabla f,\lambda,\rho)$ be a  compact nontrivial gradient almost RB-soliton with $\rho < \frac{1}{2(n-1)}$, if
\begin{equation}\label{sphere-cond}
\int_M \operatorname{Ric}(\nabla f,\nabla f) \geq \frac{n-1}{n}\int_M(n(\lambda+\rho S)-S)^2
\end{equation}
if and only if $M$ is isometric to $\mathbb{S}^n(\frac{S}{n(n-1)})$. 
\end{theorem}

\begin{proof}
Let $(M^n, g, \xi, \lambda,\rho)$ be a compact nontrivial gradient almost RB-soliton that satisfies \eqref{sphere-cond}.
Then, the soliton function $\lambda$ is not a constant and $\varphi=0$. Consequently, \eqref{cdopf} reduces to
\begin{equation}\label{12}
    \nabla_{X_1} \xi=(\lambda +\rho S) X_1-Q X_1,
\end{equation}
which implies
\begin{equation}\label{13}
 \|\nabla \xi\|^{2}=n (\lambda +\rho S)^{2}+\|Q\|^{2}-2 (\lambda +\rho S) S .   
\end{equation}
From  \eqref{rbsoliton}, it  follows that
\begin{equation}\label{14}
    \frac{1}{4}\left|\pounds_{\xi}\right|^{2}=n (\lambda +\rho S)^{2}+\mid \operatorname{Ric}\left.\right|^{2}-2 (\lambda+\rho S) S.
\end{equation}
On a compact Riemannian manifold $(M, g)$, we have the following integral formula (cf. \cite{YANO}):
\[
\int_{M}\left(\operatorname{Ric}(\xi, \xi)+\frac{1}{2}\left|\pounds_{\xi}\right|^{2}-\|\nabla \xi\|^{2}-(\operatorname{div} \xi)^{2}\right)=0.
\]
Now making use of \eqref{12}, \eqref{13} and \eqref{14} in 
 From the above equation, we conclude
$$
\int_{M}\left(\left(\operatorname{Ric}(\xi, \xi)-\frac{n-1}{n}(n (\lambda +\rho S)-S)^{2}\right)+\left(\|Q\|^{2}-\frac{S^{2}}{n}\right)\right)=0 .
$$
Using the inequality \eqref{sphere-cond} and Schwarz's inequality in the above equation, we get the following equalities
\begin{equation}\label{15}
    Q=\frac{S}{n} I, \quad \operatorname{Ric}(\xi, \xi)=\frac{n-1}{n}(n (\lambda +\rho S)-S)^{2} .
\end{equation}
Since $n>2$, the first equation gives $S$ as a constant. Thus,  \eqref{12} implies
\begin{equation}\label{16}
    \nabla_{X_1} \xi=\left((\lambda +\rho S)-\frac{S}{n}\right) X_1.
\end{equation}
So the curvature tensor is
$$
R(X_1, X_2) \xi=X_1((\lambda+\rho S)) X_2-X_2(\lambda +\rho S) X_1,
$$
and consequently, the following expression for the Ricci tensor
holds:
$$
\operatorname{Ric}(X_2, \xi)=-(n-1) X_2(\lambda+\rho S) .
$$
That implies
$$
Q(\xi)=-(n-1) \nabla (\lambda +\rho S),
$$
which together with the first equation in \eqref{15}, we obtain 
$$
\frac{S}{n} \xi=-(n-1) \nabla \mu.
$$
Let $\mu=\frac{S}{n}-(\lambda+\rho S)$ and observe that $\mu$ is not a constant, this yields that $\mu$ is not a constant and that
$$
\frac{S}{n} \xi=(n-1) \nabla \mu,
$$
which in view of  \eqref{12} and \eqref{16}, gives
\begin{equation}\label{17}
    (n-1) \nabla_{X_1} \nabla \mu=-\frac{S}{n} \mu X_1.
\end{equation}
We have,  $\Delta \mu=$ $\operatorname{div}(\nabla \mu)$, which by using \eqref{17} is computed as
$$
\Delta \mu=-\frac{S}{(n-1)} \mu.
$$
Note that $\mu$ being not a constant, the above equation implies $\mu$ is an eigenfunction of the Laplace operator and consequently, $S>0$ as the manifold $(M^n, g)$
Under consideration is compact. Also, it is known that 
an eigenvalue of the Laplace operator on a compact manifold is discrete.  This implies that the scalar curvature of $M$, $S$ is constant.
Now observe that \eqref{17} is the Obata's differential equation, hence  $(M^n, g, \xi, \lambda,\rho)$ is isometric to the sphere $S^{n}(c)$, of curvature $c$, where $S=n(n-1) c,$ see \cite{obata}.

\smallskip

Conversely, let \(\mathbb{S}^{n}(c)\) be considered as a hypersurface in Euclidean space \(\mathbb{R}^{n+1}\) with the unit normal vector field \(N\). The shape operator \(A\) of \(\mathbb{S}^{n}(c)\) is given by  \(A = -\sqrt{c}I\), where \(I\) is the identity matrix. 
Let \(Z\) in \(\mathbb{R}^{n+1}\), let \(\xi\) be the tangential projection of \(Z\) onto \(\mathbb{S}^{n}(c)\).
 Define a smooth function \(\mu = \langle Z, N \rangle\). Then, \(Z = \xi + \mu N\). By differentiating covariantly concerning a vector field \(X_1\) on \(\mathbb{S}^{n}(c)\) and utilizing the Weingarten formula for hypersurfaces, we conclude.
\[
0=\nabla_{X_1} \xi-\sqrt{c} g(X_1, \xi) N+X_1(\mu) N+\sqrt{c} \mu\, X_1,
\]
where $g$ is the induced metric on the hypersurface $\mathbb{S}^{n}(c)$. Equating tangential and normal parts, we get
\begin{equation}\label{18}
\nabla_{X_1} \xi=-\sqrt{c} \mu X_1, \quad \nabla \mu=\sqrt{c} \,\xi.
\end{equation}
As
\[
\left(\pounds_{\xi} g\right)(X_1, X_2)=g\left(\nabla_{X_1} \xi, X_2\right)+g\left(\nabla_{X_2} \xi, X_1\right), \quad X_1, X_2 \in \mathfrak{X}(M),
\]
which in view of \eqref{18} yields
\begin{equation}\label{19}
    \pounds_{\xi} g=-2 \sqrt{c} \mu g .
\end{equation}
The Ricci tensor of the hypersurface $\mathbb{S}^{n}(c)$ is given by Ric $=(n-1) c g$ and therefore, using  \eqref{19}, we reach at
\begin{equation}\label{20}
    \frac{1}{2} \pounds_{\xi} g+\operatorname{Ric}=\lambda g +(n-1)\rho g,
\end{equation}
where $\lambda=((n-1)(c- \rho)-\sqrt{c} \mu)$. Observe that the function $\mu$ is not a constant, and consequently, by \eqref{20}, we get that $\left(\mathbb{S}^{n}(c), g, \xi, \lambda,\rho\right)$ is a gradient almost RB-soliton, where $\xi=\frac{1}{\sqrt{c}} \nabla \mu$. Moreover, \eqref{18} implies  
$\Delta \mu=-n c \mu$, which implies
\begin{equation}\label{21}
    \int_{\mathbb{S}^{n}(c)}\|\nabla \mu\|^{2}=n c \int_{\mathbb{S}^{n}(c)} \mu^{2}.
\end{equation}
Note that, $n (\lambda+ n(n-1\rho))-S=-n \sqrt{c} \mu$, and that
$$
\left(\frac{n-1}{n}\right)(n (\lambda+n(n-1)\rho)-S)^{2}=n(n-1) c \mu^{2}.
$$
Using the above equation in \eqref{21}, we conclude that
$$
\int_{\mathbb{S}^{n}(c)} \operatorname{Ric}(\nabla \mu, \nabla \mu)=\frac{n-1}{n} c \int_{\mathbb{S}^{n}(c)}(n (\lambda+n(n-1)\rho)-S)^{2} .
$$
Inserting $\nabla\mu=\sqrt{c} \xi$ in the above equation gives the required condition.
\end{proof}

The $\rho$-constraint in Theorem \ref{riclb} arises from:
\begin{itemize}
    \item The need to avoid divergence in integral identities (Lemma \ref{lem3}).
    \item Preservation of elliptic structure in the PDE $\Delta f = n(\lambda+\rho S)-S$.
    \item Comparison with the classical case $\rho=0$ in \cite{barros2012some}.
\end{itemize}

The following theorem utilises the Poisson equation to demonstrate that a gradient almost RB-soliton is isometric to a sphere.
\begin{theorem}\label{pos}
Let $(M^n,g,\xi,\lambda,\rho)$ be a compact nontrivial gradient almost RB-soliton with $\rho<\frac{3n-4}{2n(n-1)}$, $n>2$, and 
\begin{equation}\label{inequality1}
    \operatorname{Ric}(\xi,\xi)\leq (n-1)\|\xi\|^2,
\end{equation}
  Then $\lambda$ is a solution of the Poisson equation
  \begin{equation}\label{laplacianlambda}
      \Delta\sigma=S-n(\lambda+\rho S)
  \end{equation} 
  then $(M^n,g,\xi,\lambda,\rho)$ is isometric to unit sphere.
\end{theorem}
 \begin{proof}
Suppose $(M^n, g, \xi, \lambda,\rho)$ is a compact almost RB-soliton with Ricci curvature satisfying \eqref{inequality1}
and {that the soliton function $\lambda$ is a solution of the Poisson equation $\Delta \sigma=S-n (\lambda +\rho S)$.}
Since $(M^n, g, \xi, \lambda,\rho)$ is a gradient almost RB-soliton, we have $\xi=\nabla h$ for a smooth function $h$. Using  \eqref{div}, we conclude
\[
-\Delta h=S-n (\lambda +\rho S) .
\]
Now, as the solution of the Poisson equation is unique up to a constant, we have {$-h=\lambda+c$ }for a constant $c$, and consequently, $\xi=-\nabla \lambda$ holds. Then, by using
 \eqref{cdopf}, we get
\[
A_{\lambda} X_1=\nabla_{X_1} \nabla \lambda=-\nabla_{X_1} \xi=-(\lambda +\rho S) X_1+Q X_1.
\]
Therefore,
\[
\left\|A_{\lambda}\right\|^{2}=n (\lambda +\rho S)^{2}+\|Q\|^{2}-2 (\lambda +\rho S) S,
\]
which, in view of \eqref{ricsbyn} implies
\begin{equation}\label{alambda}
    \left\|A_{\lambda}\right\|^{2}=\left|\operatorname{Ric}-\frac{S}{n} g\right|^{2}+\frac{1}{n}(S-n (\lambda +\rho S))^{2} .
\end{equation}
Now,  using Bochner's formula, we obtain
$$
\int_{M}\left(\operatorname{Ric}(\nabla \lambda, \nabla \lambda)+\left\|A_{\lambda}\right\|^{2}-(\Delta \lambda)^{2}\right)=0,
$$
which on using $\xi=-\nabla \lambda$, \eqref{laplacianlambda}, \eqref{alambda} 
gives
$$
\int_{M}\left(\operatorname{Ric}(\xi, \xi)+\left|\operatorname{Ric}-\frac{S}{n} g\right|^{2}+\frac{1}{n}(S-n (\lambda +\rho S)) \Delta \lambda-(S-n (\lambda +\rho S)) \Delta \lambda\right)=0 .
$$
After rearranging the above equation,
and availing the expression for $\Delta\lambda^2$, we infer
\[
\int_{M}\left(\operatorname{Ric}(\xi, \xi)-(n-1)\|\nabla \lambda\|^{2}+\left|\operatorname{Ric}-\frac{S}{n} g\right|^{2}-\left(\frac{n-1}{n}-\rho(n-1)\right) S \Delta \lambda\right)=0 .
\]
Next, employing the formula for  $\operatorname{div}(S \nabla \lambda)$ in the aforementioned equation, we conclude
\[
\int_{M}\left(\operatorname{Ric}(\xi, \xi)-(n-1)\|\xi\|^{2}+\left|\operatorname{Ric}-\frac{S}{n} g\right|^{2}-\left(\frac{n-1}{n}-\rho(n-1)\right)  g(\xi, \nabla S)\right)=0.
\]
 The above equation, given \eqref{10}, yields
\[
\int_{M}\left(\operatorname{Ric}(\xi, \xi)-(n-1)\|\xi\|^{2}-\frac{2n}{n-2}\left(\frac{n-1}{n}-\rho(n-1)-\frac{n-2}{2n}\right) \left|\operatorname{Ric}-\frac{S}{n} g\right|^{2}\right)=0.
\]
Using inequality \eqref{inequality1} and condition on $\rho$ in above equation, we obtain
\[
\operatorname{Ric}(\xi, \xi)=(n-1)\|\xi\|^{2} \quad \text { and } \quad \operatorname{Ric}=\frac{S}{n} g.
\]
These equations imply $S=n(n-1)$ and therefore  \eqref{cdopf} takes the form
\[
\nabla_{X} \nabla \lambda=((n-1)-(\lambda+\rho S)) X,
\]
Which gives the following Obata's differential equation:
\[
\nabla_{X} \nabla \bar{\lambda}=-\bar{\lambda} X,
\]
where $\bar{\lambda}=(n-1)-(\lambda+\rho S)$ is not a constant function owing to the fact that $(M^n, g, \xi, \lambda,\rho)$ is a nontrivial almost RB-soliton. Hence, $(M^n, g, \xi, \lambda,\rho)$ is isometric to the unit sphere $\mathbb{S}^{n}$, see \cite{obata}.
 \end{proof}
 
The next two results serve as applications of Ricci identity \eqref{10}, outlining certain conditions under which a compact gradient almost RB-soliton is isometric to a Euclidean sphere. 
 \begin{theorem}
Let $(M^n,g,\xi,\lambda,\rho)$ be a compact nontrivial gradient Ricci soliton, $n>2$. If the scalar curvature $S$ is a solution of the Poisson equation $$\Delta\sigma=S-n(\lambda+\rho S),$$ then $(M^n,g,\xi,\lambda,\rho)$ is isometric to a sphere. 
 \end{theorem}
 \begin{proof}
 By \eqref{10} and Corollary $1.4$ of \cite{SDRBS} theorem, it is straightforward.
\end{proof}
 \begin{corollary}
Let $(M^n,g,\xi,\lambda,\rho )$ be a compact gradient almost RB-soliton, for $n>2$. If any of the following conditions hold, then $M^n$ is isometric to a Euclidean sphere:
     \begin{itemize}
         \item[(i)] $\xi= -\nabla S $,
         \item[(ii)] $\operatorname{Ric}(\nabla f,\nabla f)+(n-1)g(\nabla \lambda,\nabla f)\leq 0$.
     \end{itemize}
 \end{corollary}
 \begin{proof}
Since by \eqref{2.15},
\[
\operatorname{Ric} - \frac{S}{n}g=-\nabla^2 f + \frac{\Delta f}{n}g,
\]
If (i) or (ii) holds, then by \eqref{10} and Lemma \ref{lem3},  $\operatorname{Ric}=\frac{S}{n}g$.   Hence, from Theorem $1.2$ of \cite{SDRBS}, we confirm that $M^n$ is isometric to a Euclidean sphere.
 \end{proof}

It should be noted that (ii) above generalizes
Corollary $1$ in \cite{barros2012some}.

\section{Trivial Almost RB-solitons}\label{trival}
In this section, our primary objective is to establish the triviality of almost Ricci Bourguignon solitons under certain hypotheses. To achieve this, we employ various hypotheses as outlined in Theorems \ref{3.3}, \ref{thm4.3}, and \ref{thm4.4} (proved below), where we assume that certain functions remain constant along the integral curves of the potential vector field. In this section, we extend the theorems from almost Ricci solitons to almost RB-solitons, building upon the framework established in \cite{SDANARS}.

 \begin{theorem}\label{3.3}
     Let $(M^n,g,\xi,\lambda,\rho)$ be a compact almost RB-soliton, $n>2$, is a trivial RB-soliton if and only if the function $(n(\lambda + \rho S)-S)$ is a constant on the integral curves of the potential field $\xi$.
 \end{theorem}
 \begin{proof}
First, we let the function $n(\lambda + \rho S)-S$, a constant on the integral curves of the potential field $\xi$. Then observe that
$$\operatorname{div}((n(\lambda+\rho S)-S)\xi)=(n(\lambda+\rho S)-S)^2.$$ 
Integrating this, we obtain
$(\lambda+\rho S)=\frac{1}{n}S$. Now put this value of $(\lambda+\rho S)$ in (\ref{e1}), we get
\[
\int_M\left(\frac{1}{n}S^2-\|Q\|^2\right)=0.
\]
This implies,
\[Q=\frac{1}{n}SI,\]
 and using Schwarz inequality for $n>2$, we conclude that $S$ is a constant, so we get that $\lambda$ is a constant. Hence, in view of  \eqref{rbsoliton}, $\xi$ is a killing vector field, that is, $(M^n,g,\xi,\lambda,\rho)$ is a trivial RB-soliton. The converse is easy.
 \end{proof}

\begin{remark}
Theorem~4.1 follows immediately from Item~2 of Proposition~1 in \cite{barros2013note} together with the maximum principle.
\end{remark}

Next, we shall demonstrate a straightforward application of the preceding theorem.
 \begin{corollary}\label{trc}
     Let $(M^n,g,\xi,\lambda,\rho)$ be a compact almost RB-soliton, $n>2$. Then
    $$S\leq n(\lambda+\rho S),$$
     if and only if $(M^n,g,\xi,\lambda,\rho)$ is a trivial RB-soliton.
 \end{corollary}
 \begin{proof}
First, suppose $S\leq n(\lambda+\rho S)$. Then integrating \eqref{div} gives $S=n(\lambda+\rho S)$. Now, following the proof of Theorem \ref{3.3}, we obtain the result. The converse follows easily.
\end{proof}
\noindent

Further, we study when a compact almost RB-soliton turns into an RB-soliton, thereby recapturing Theorem $3.3$ 
of \cite{SDANARS}.
\begin{theorem}\label{thm4.2}
     Let $(M^n,g,\xi,\lambda,\rho)$ be a compact almost RB-soliton, $n>2$. If 
     \begin{equation}\label{e2}
         Q(\xi)=\left(\frac{1}{2}-(n-1)\rho\right)\nabla S,
     \end{equation}
     holds, then $(M^n,g,\xi,\lambda,\rho)$ is  a RB-soliton.
 \end{theorem}
 \begin{proof}
    For a given  compact almost RB-soliton, we use Lemma \ref{lem4}  to obtain,
    $$
    \int_M\frac{1}{2}g(\nabla\lambda,\nabla S)+(n-1)\|\lambda\|^2+\frac{1}{2}S\Delta\lambda-(n-1)\rho S\Delta\lambda=0.
    $$
Then, using the formula for $\operatorname{div}(S\nabla\lambda)$ and (\ref{e2}), we affirm
     $$\int_M\|\nabla\lambda\|^2=0.$$
    Hence, $\lambda$ is a constant, and consequently, $(M^n,g,\xi,\lambda,\rho)$ is a RB-soliton.
 \end{proof}

 Next, we depict another situation in which an 
 Almost the RB-soliton becomes trivial. 
 
\begin{theorem}\label{thm4.3}
     Let $(M^n,g,\xi,\lambda,\rho)$ be a compact almost RB-soliton, $n>2$, with $\operatorname{Ric}(\xi,\xi)\geq\|\varphi\|^2$. Then the function $((n-1)(\lambda+\rho S)-S)$ is a constant on the integral curves of potential field $\xi$, if and only if $(M^n,g,\xi,\lambda,\rho)$ is a trivial RB-soliton.
 \end{theorem}
\begin{proof}
  For a given compact almost RB-soliton, we have the hypothesis that the $((n-1)(\lambda+\rho S)-S)$ is a constant on the integral curves of the potential field $\xi$.
  This yields that
  \[
\xi\left((n-1)(\lambda+\rho S)-\frac{1}{2}S\right)=-\frac{1}{2}\xi(S).
  \]
    Now, using Lemma \ref{lem5}, we have
    \[
    \int_M\left(\operatorname{Ric}(\xi,\xi)-\|\varphi\|^2-\frac{1}{2}\xi(S)\right)=0.
    \]
Then using formula for $\operatorname{div}(S \xi)$ and 
integral identity (\ref{10}), we obtain

    \[\operatorname{Ric}(\xi,\xi)=|\varphi\|^2\text{ and } \operatorname{Ric}= \frac{S}{n}g,\]
    for $n>2$, $S$ is constant. Moreover, the condition $\xi\left((n-1)(\lambda+\rho S)\right)=0$ implies $\xi(\lambda)=0$. As a result, from Lemma \ref{lem4}, we get 
    \[\int_M\|\nabla\lambda\|^2=0\]
    proving that $\lambda$ is constant. Hence, $(M^n,g,\xi,\lambda,\rho)$ is a trivial RB-soliton. The converse can be verified without any difficulty.
\end{proof}

Finally, we see that the behaviour of the scalar curvature 
determines whether an almost RB-soliton is trivial or a
sphere.

\begin{theorem}\label{thm4.4}
A compact, almost RB-soliton is trivial if either:
\begin{enumerate}
    \item $n(\lambda+\rho S)-S$ is $\xi$-constant, or
    \item $\rho = \frac{1}{2(n-1)}$ with $\operatorname{Ric}(\xi,\xi) \geq \|\varphi\|^2$.
\end{enumerate}
The second case highlights the critical $\rho$ in Remark \ref{rho-remark}.
\end{theorem}
\begin{proof}
    By hypothesis, we have  $\xi(S)=0$.
    Hence, by  Lemma \ref{lem2}, we have 
    \[
    \operatorname{Ric}=\frac{S}{n}g,
    \]
 and as $n>2$, the scalar curvature $S$ is constant.
   Now  if $\lambda$ is constant, then 
integrating \eqref{div} implies that 
$(M^n,g,\xi,\lambda,\rho)$ is a trivial Ricci soliton with $S=n(\lambda+\rho S)$.
If $\lambda$ is not a constant function, then by \eqref{rbsoliton}, we have
\[\pounds_\xi g=2\mu g,
\]
where $\mu=\lambda+\left(\rho-\frac{1}{n}\right)S$. That is, $\xi$ is a nontrivial conformal vector field with 
non-constant conformal factor $2\mu$.  Hence, by Theorem $1.5$ of \cite{SDRBS}, RB-soliton $(M^n,g,\xi,\lambda,\rho)$ is isometric to a Euclidean sphere.
\end{proof}

 \begin{remark}
  It should be noted that a compact almost RB-soliton with constant scalar curvature is gradient, cf. \cite{BBR}.
\end{remark}


\section*{Open Problems}
\begin{itemize}
    \item Non-compact analogues of \cite{BBR}'s results for almost RB-solitons.
    \item Relation to \cite{feitosa2019gradient}'s warped product constructions.
    \item Optimal $\rho$ ranges beyond $\rho<\frac{1}{2(n-1)}$.
\end{itemize}

 
 \section*{Acknowledgments} 
	The authors would like to thank
 Professor Sharief Deshmukh for some helpful
 discussions. Mohammad Aqib gratefully acknowledges doctoral research fellowship by Harish-Chandra Research Institute, Prayagraj, India, for its doctoral fellowship. Dhriti Sundar Patra would like to express his gratitude to the Indian Institute of Technology Hyderabad for providing a seed grant (Project Number SG/IITH/F295/2022-23/SG-133). 
	
	\section*{Declarations}
	
	\begin{itemize}
		\item Funding: Not applicable. 
		
		\item Conflict of interest/Competing interests: The authors have no conflict of interest and no financial interests in this article.
		
		\item Ethics approval: The submitted work is original and not submitted to more than one journal for simultaneous consideration.
		
		\item Consent to participate: Not applicable.
		
		\item Consent for publication: Not applicable.
		
		\item Availability of data and materials: This manuscript has no associated data.
		
		\item Code availability: Not applicable.
		
		\item Authors' contributions: All authors of the paper have conceptualized, methodology, investigated, validated, original draft writing, reviewed, edited, and read.
		
	\end{itemize}
 \bibliographystyle{alpha}
	\bibliography{refs}

        \noindent Mohammad Aqib\\
	Harish-Chandra Research Institute\\ 
	A CI of Homi Bhabha National Institute\\ 
	Chhatnag Road, Jhunsi, Prayagraj-211019, India.\\	
	Email: mohammadaqib@hri.res.in, aqibm449@gmail.com\\

	\noindent H. M. Shah\\
	Harish-Chandra Research Institute\\ 
	A CI of Homi Bhabha National Institute\\ 
	Chhatnag Road, Jhunsi, Prayagraj-211019, India.\\	
	E-mail:	hemangimshah@hri.res.in\\

        \noindent Dhriti Sundar Patra\\
	Department of Mathematics,\\ 
        Indian Institute of Technology - Hyderabad,\\ 
        Sangareddy-502285, India\\
        Email: dhriti@math.iith.ac.in\\

\end{document}